\documentclass{article}

\usepackage{arxiv}

\usepackage[utf8]{inputenc} 
\usepackage[T1]{fontenc}    
\usepackage{hyperref}       
\usepackage{url}            
\usepackage{booktabs}       
\usepackage{amsfonts}       
\usepackage{nicefrac}       
\usepackage{microtype}      

\usepackage{graphicx}
\usepackage{amsmath}
\usepackage{amsthm}
\usepackage[multiple]{footmisc}

\newtheorem{theorem}{Theorem}
\newtheorem{lemma}{Lemma}

\usepackage[linesnumbered,ruled,vlined]{algorithm2e}
\SetKwInput{KwInput}{\texttt{Input}}
\SetKwInput{KwOutput}{\texttt{Output}}
\SetKwSty{texttt}

\makeatletter
\renewcommand{\p@subsection}{}
\makeatother

\title{Tensor-Train Numerical Integration of Multivariate Functions with Singularities}

\begin{document}

\author{
    Lev~I.~Vysotsky$^{1,2,3}$
    \\
    \texttt{vysotskylev@yandex.ru}
    
    \And
    
    Alexander~V.~Smirnov
    $^{4,3}$
    \\
    \texttt{asmirnov80@gmail.com}
    
    \And
    
    Eugene~E.~Tyrtyshnikov
    $^{5,4,3}$\\
    \texttt{eugene.tyrtyshnikov@gmail.com}
}

\footnotetext[1]{HSE University}
\footnotetext[2]{Faculty of Computational Mathematics and Cybernetics, Moscow State University}
\footnotetext[3]{Moscow Center for Fundamental and Applied Mathematics}
\footnotetext[4]{Research Computing Center, Moscow State University}
\footnotetext[5]{Marchuk Institute of Numerical Mathematics of Russian Academy of Sciences}

\maketitle

\begin{abstract}
Numerical integration is a classical problem emerging in many fields of science.
Multivariate integration cannot be approached with classical methods due to the exponential growth of the number of quadrature nodes.
We propose a method to overcome this problem.
Tensor-train decomposition of a tensor approximating the integrand is constructed and used to evaluate a multivariate quadrature formula.
We show how to deal with singularities in the integration domain and conduct theoretical analysis of the integration accuracy.
The reference open-source implementation is provided.
\end{abstract}


\section{Introduction}

Numerical integration is a classical problem that emerges in many fields of science.
This paper was inspired by a physical application: a numerical approach to the evaluation of Feynman integrals based on so-called sector decomposition (see, for example, \cite{Bogner2007}).

The classical naive approach (see e.g.~\cite{Piessens1983}) to multivariate integration meets the so-called curse of dimensionality.
Let us for simplicity consider the case of integrating over a unitary hypercube of dimension $d$.
In the classical approach to multivariate integration one  fixes $n$ nodes on each of the coordinate axes and sums the values of the integrand in each of the $n^d$ points with weights prescribed by the used quadrature rule (i.e. Gaussian or Simpson's).
However whatever rule is chosen, the $n^d$ grows exponentially with $d$, so the classical approaches turn out to be completely inefficient in high-dimensional case.


There exist different approaches to the multivariate integration avoiding the curse of dimensionality, both deterministic and nondeterministic.
The nondeterministic methods are usually called Monte-Carlo algorithms.
Historically the first method of this class is the Vegas algorithm originally presented in \cite{Lepage1978, Lepage1980}.
It uses importance sampling for variance reduction.
One might pose a questions why there is a place for randomness in a task when a deterministic result is required.
The idea is that the algorithms randomly disseminating sampling points have almost a hundred percent chance of hitting the ``special'' regions of integration domain.
On the contrary, a deterministic improperly chosen method might skip such regions thus leaning to an incorrect result.
Hence while using random point sampling those algorithms are supposed to end in almost the same result each time with the difference being negligible and smaller than the error estimate.

There are other known nondeterministic algorithms such as Suave combining the advantages of two popular methods: importance sampling by Vegas and subregion sampling in a manner similar to
Miser~\cite{Press1990}.
By dividing into subregions, Suave manages to a certain extent to get around Vegas' difficulty to adapt its weight function to structures not aligned with the coordinate axes.
However, the disadvantage of this method lies it the obligation of the user to provide some extra information of the function structure which is not possible in a general case.

There are also deterministic methods known such as Cuhre using a cubature rule for subregion estimation in a globally adaptive subdivision scheme~\cite{Berntsen1991}.
However, such approaches still suffer from the curse of dimensionality and might compete with the Monte-Carlo class only at lower dimensions such as 2--4, while in the physical problem that inspired this work one might expect up to 10 dimensions or even more.

The development of the Monte-Carlo approach led to emergence of the so-called Quasi Monte Carlo (QMC) integration methods~\cite{Schurer2004, Dick2013}.
The review of them can be found in~\cite{Borowka2017}, where one of them was applied for the evaluation of Feynman integrals.
For a description of a modern implementation of the QMC algorithm see~\cite{Borowka2019}.

A completely new approach to multivariate integration based on the tensor-train decomposition was originally suggested in~\cite{Oseledets2010}.
The idea is to use low-parametric approximation (tensor-train, or TT, decomposition) to the multivariate function to overcome the curse of dimensionality of the naive multivariate integration.
The proposed decomposition can be constructed using the function values computed in moderate number of points.
It also allows to efficiently calculate the multidimensional weighted sum implied by quadrature rule.
However, in~\cite{Oseledets2010} the authors provided only a ``proof-of-concept'' of TT integration and applied it to some model functions.
Recently, in~\cite{Dolgov2020} the method was used to compute high-dimensional (up to 1024 dimensions) integrals and was shown to outperform the stochastic methods.
Moreover, the method seems applicable to very high-dimensional integrals (up to hundreds of thousands of dimensions) like those found e.g. in condensed matter physics, as the computational complexity grows only linearly with the number of dimensions.

In this paper we provide a complete description of the tensor-train-based multivariate integration of functions with singularities.
A configurable and practically applicable algorithm is presented and its reference open-source implementation is provided.
We conduct the theoretical error analysis and show how to deal with some types of singularities in the integration domain.
The algorithm was applied to the computation of a Feynman integral and shown to perform better than the Monte-Carlo-based method.

The rest of the paper is organized as follows.
In Section~\ref{sec:definitions} we introduce concepts and facts necessary for the formulation of the method.
Section~\ref{sec:integration-high-level} contains the bird's eye view of the integration algorithm.
In Sections~\ref{sec:matrix-cross} and~\ref{sec:tt-cross} the details of TT cross approximation approach are provided along with some supporting arguments.
In Section~\ref{sec:tolerance} we consider the problem of choosing the approximation tolerance given the limit on the number of function evaluations.
Our approach to accurate integration of functions with singularities is described in Section~\ref{sec:singularities}.
The theoretical analysis of the integration accuracy is provided in Section~\ref{sec:analysis}.
Finally, the discussion of the implementation details and usage of the library can be found in Section~\ref{sec:implementation}.
There we also provide a comparison of the TT and Monte-Carlo-based methods applied to a computation of a Feynman integral.

\section{Definitions and preliminary facts}\label{sec:definitions}

Let $\mathbb{C}^{n_1 \times\cdots \times n_d}$ denote the set of $d$-dimensional tensors (i.e. arrays) over field $\mathbb{C}$ with mode sizes $n_1, \dots, n_d$.
For a given tensor $A \in \mathbb{C}^{n_1 \times\cdots \times n_d}$ denote $A(i_1, \dots, i_d)$ its element with indices $i_1, \dots, i_d$.
We will use zero-based indexation, i.e. $0 \le i_k \le n_k - 1$.

If in indexation expression $A(i_1, \dots, i_d)$ we replace variables $i_{k_1}, \dots, i_{k_s}$ with colons we obtain an $s$-dimensional tensor $B \in \mathbb{C}^{n_{k_1}, \dots, n_{k_s}}$ with elements
$
B(i_{k_1}, \dots, i_{k_s}) = A(i_1, \dots, i_d).
$
The tensor $B$ is thus determined by ``non-colon'' variables.
As an example, consider a 3-dimensional tensor $A \in \mathbb{R}^{2 \times 3 \times 4}$ with elements $A(i, j, k) = i + j + k$.
Then $B := A(:, :, 2)$ is a $2\times 3$ matrix with elements $B(i,j) = i + j + 2$, and $v := A(1, 1, :)$ is a vector of length $4$ with elements $v(k) = k + 2$.
The notation is adopted from Matlab and Python programming languages.
Other useful adoptions are the $\mathsf{vec}$ and $\mathsf{reshape}$ functions.
The former one makes a ``long'' vector:
\[
\mathsf{vec}(A) \in \mathbb{C}^{n_1\cdots n_d},~\mathsf{vec}(A)(i_1 + i_2 n_1 + \dots + i_d n_1 \dots n_{d-1}) = A(i_1, \dots, i_d),
\]
The $\mathsf{reshape}$ function takes a tensor $A \in \mathbb{C}^{n_1\times\dots\times n_d}$ and new shape $m_1\times\dots\times m_D$ and returns tensor $B \in \mathbb{C}^{m_1\times\dots\times m_D}$ such that $\mathsf{vec}(A)=\mathsf{vec}(B)$ (naturally, $n_1\cdots n_d$ must be equal to $m_1\cdots m_D$).
For each $k=1,\dots,d-1$ the $k$-th unfolding matrix $A_k$ of a given tensor $A$ is defined as 
\[
A_k = \mathsf{reshape}(A, (n_1\cdots n_k)\times(n_{k+1}\cdots n_d)).
\]

We say that a tensor $A\in\mathbb{C}^{n_1 \times\cdots \times n_d}$ admits \emph{tensor-train (TT) decomposition}~\cite{Oseledets2011} if
\begin{equation}\label{eq:TT}
A(i_1, \dots, i_d) = \sum_{\alpha_0=0}^{r_0-1}\dots \sum_{\alpha_{d}=0}^{r_{d}-1} G_1(\alpha_0, i_1, \alpha_1) G_2(\alpha_1, i_2, \alpha_2) \dots G_d(\alpha_{d-1}, i_d, \alpha_d)
\end{equation}
for some natural numbers $r_1, \dots, r_{d-1}$ and 3-dimensional tensors $G_k \in \mathbb{C}^{r_{k-1}\times r_k \times r_{k+1}}$.
Note that we force $r_0$ and $r_d$ to be equal to 1, thus in equation~\eqref{eq:TT} variables $\alpha_0$ and $\alpha_d$ are ``dummy'' ones and tensors $G_1$ and $G_d$ are actually 2-dimensional.
Tensors $G_k$ are called the \emph{cores} and numbers $r_k$ the \emph{compression ranks} (or TT-ranks) of decomposition~\eqref{eq:TT}.
An equivalent way to define tensor-train decomposition is
\begin{equation}\label{eq:TT-matrices}
A(i_1, \dots, i_d) = G_1(:, i_1, :)\dots G_d(:, i_d, :).
\end{equation}
Note that each multiplier $G_k(:, i_k, :)$ above is a matrix (one index is fixed while the two other are ``free'') of size $r_{k-1}\times r_{k}$.
The product is thus correctly defined, moreover, as $r_0 = r_d = 1$ it has size $1 \times 1$, so it is indeed a number from $\mathbb{C}$.
It can be easily checked that definitions~\eqref{eq:TT} and~\eqref{eq:TT-matrices} are equivalent.
The latter is best suited for computations as it exploits fast BLAS level-3 operation of matrix multiplication.
The fact that $A$ admits TT-decomposition with cores $G_1,\dots,G_d$ is denoted $A = \mathsf{TT}(G_1,\dots,G_d)$.

Obviously, the storage required for a given TT-representation of a tensor is proportional to $\sum_k n_k r_{k-1} r_k$ or if all mode sizes are bounded from above by $n$ and compression ranks by $r$, the storage can be estimated as $O(dnr^2)$.

For a given tensor $A\in\mathbb{C}^{n_1 \times\cdots \times n_d}$ we can (at least theoretically) find the smallest compression ranks among all the possible TT-decompositions of it.
These compression ranks are equal to corresponding ranks of unfolding matrices: $r_k =\mathrm{rank}\,A_k$.
The existence of a TT-decomposition with compression ranks equal to $\mathrm{rank}\,A_k$ is proven in~\cite{Oseledets2011}.
The non-existence of TT-decompositions with smaller compression ranks follows from the fact that any TT-decompositions with TT-ranks $r_k$ yields a skeleton decomposition of unfolding matrix $A_k$ with $r_k$ summands.

In practice, however, we are interested in low rank TT-approximations, i.e. tensors $A'$ close to $A$ in some norm that admit TT-decomposition with low ranks.
In~\cite{Oseledets2010} it was proven that if unfolding matrices are approximated by rank-$r_k$ matrices $B^{(k)}$:  
$
\|A_k - B^{(k)}\|_F \le \varepsilon_k,
$
$k=1,\dots,d-1$,
then there exists a tensor $A'$ with compression ranks $r_k$ satisfying
\begin{equation}\label{eq:unfolding-approximation}
\|A' - A\|_F^2 \le \sum_{k=1}^{d-1}\varepsilon_k^2
\end{equation}
(here $\|A\|_F$ denotes Frobenius norm of a tensor, i.e. square root of  $\sum_{i_1,\dots,i_d}|A(i_1,\dots,i_d)|^2$).

If tensors $A$ and $B$ are represented in TT-format (i.e. some TT-decomposition is known for each of them), several basic operations may be performed directly using such representation~\cite{Oseledets2011}.
Addition, Hadamard (element-wise) product, dot product and multidimensional contraction are the examples of such operations.
The dot product of tensors $\langle A, B \rangle$ is defined as
$
\langle A, B \rangle = \sum_{i_i, \dots, i_d} (\overline{A}\circ B)(i_1, \dots, i_d),
$
where $\overline{A}$ is tensor whose elements are complex conjugates of elements of $A$: $\overline{A}(i_1,\dots,i_d) = \overline{A(i_1,\dots,i_d)}$ and
$A \circ B$ denotes element-wise product of $A$ and $B$.
Multidimensional contraction of tensors $A$ with columns $u_k \in \mathbb{C}^{n_k}$ is defined as  $W = \langle \overline{A}, U\rangle$, where $U$ is a tensor with compression ranks equal to 1 and cores $G_k(0, :, 0) = u_k$.

\section{The integration algorithm}
\subsection{High-level description}\label{sec:integration-high-level}
Suppose we are given a function $f: \mathbb{R}^d \to \mathbb{C}$ and need to compute (approximately) the integral
\begin{equation}\label{eq:I(f)}
I(f) = \int_{0}^1\dots\int_0^1f(x_1,\dots,x_d)dx_1\cdots dx_d.    
\end{equation}
We assume that some nodes $\{x_k^{(i_k)}\}_{i_k=0}^{n_k-1}$, $x_k^{(i_k)} \in \mathbb{R}$, and weights $\{w_k^{(i_k)}\}_{i_k=0}^{n_k-1}$, $w_k^{(i_k)} \in \mathbb{R}$, are fixed
for each axis $x_k$.
Consider the $d$-dimensional grid $X$:
\[
X := \{x_1^{(i_1)}\}_{i_1=0}^{n_1-1} \times \dots \times \{x_1^{(i_d)}\}_{i_d=0}^{n_d-1}
\]
and weight tensor $W \in \mathbb{C}^{n_1\times\cdots\times n_d}$ with elements
$
W(i_1, \dots, i_d) = w_1^{(i_1)} \cdots w_d^{(i_d)}.
$
We will approximate integral~\eqref{eq:I(f)} with the sum
\begin{equation}
    S_{X,W}(f) := \sum_{i_1=0}^{n_1-1}\dots\sum_{i_d = 0}^{n_d-1} f\left(x_1^{(i_1)}, \dots, x_d^{(i_d)}\right) W(i_1, \dots, i_d).
\end{equation}
If we denote $A_{X,f}$ the tensor from $\mathbb{C}^{n_1\times\cdots\times n_d}$ with elements $f\big(x_1^{(i_1)}, \dots, x_d^{(i_d)}\big)$, the introduced sum can be written as $S_{X,W}(f) = \langle \overline{A_{X,f}}, W \rangle$.
But even if weights $w_k^{(i_k)}$ are chosen wisely and the sum $S_{X,W}(f)$ accurately approximates $I(f)$, this approach is not free from the ``curse of dimensionality'' as direct computation of $S_{X,W}(f)$ requires $n_1\cdots n_d$ evaluations of $f(x_1,\dots,x_d)$.
Let us assume, however, that $A_{X,f}$ admits an approximation $A'$ with moderate TT-ranks $r_k$.
Then the value $S' = \langle \overline{A'}, W\rangle$ can be computed as multidimensional contraction (as tensor $W$ has compression ranks equal to 1) with complexity $O\left(\sum_k n_k r_k r_{k-1}\right)$ which is polynomial in mode sizes $n_k$ and compression ranks $r_k$.

The main difficulty consists in obtaining a good approximation of $A_{X,f}$ without computing all (or even a significant ratio of) its entries.
For many functions arising from real-world applications this is an achievable target.

Algorithm~\ref{alg:tt-cross} presents the pseudocode for constructing a TT-approximation of a given black-box (i.e. with specified procedure of computing an element with given indices) tensor.
It is described in detail in Section~\ref{sec:tt-cross}.
The workhorse of the method are the subroutines \texttt{TTCrossLeftToRight} (see Algorithm~\ref{alg:cross-tt-pass}) and \texttt{TTCrossRightToLeft} (pseudocode is not presented due to its extreme similarity).
After each pass we obtain a (hopefully) improved approximation to $A$.
During one pass for each dimension $k$ a carefully selected submatrix of the unfolding matrix $A_k$ is approximated by a cross consisting of $r_k$ rows and $r_k$ columns.
For this we use the approach which pseudocode is presented in Algorithm~\ref{alg:matrix-cross}  and which is described in detail in Section~\ref{sec:matrix-cross}).
The obtained cross is used to construct the $k$-th TT core of the approximation.
The lines~\ref{alg:tt-cross-test} and~\ref{alg:tt-cross:effective} of Algorithm~\ref{alg:tt-cross} are responsible for choosing the approximation accuracy according to the user-specified limit on the number of tensor element evaluations (for details refer to Section~\ref{sec:tolerance}).

\begin{algorithm}\label{alg:tt-cross}
\KwInput{(black-box) tensor $A \in \mathbb{C}^{n_1\times\dots\times n_d}$, evaluation count limit $N_{\mathrm{lim}}$}
\KwOutput{TT cores $G_1,\dots, G_d$ of approximation $A'$ of $A$}
    Randomly choose sets of column multi-indices $\widehat{\mathcal{J}}_1,\dots, \widehat{\mathcal{J}}_{d-1}$\;
    $G_1,\dots, G_d, \mathcal{I}_1,\dots,\mathcal{I}_{d-1}\gets \mathtt{TTCrossLeftToRight}(A, \widehat{\mathcal{J}}_1,\dots, \widehat{\mathcal{J}}_{d-1}, \varepsilon_{\mathrm{test}})$
    \label{alg:tt-cross-test}
    \tcp{measure $N_{\mathrm{test}}$}
    $\varepsilon_{\mathrm{eff}} \gets $ effective tolerance
    (see Section~\ref{sec:tolerance})
    \label{alg:tt-cross:effective}\;
    \Repeat{
       $\Big\|\mathsf{TT}(H_1,\dots,H_d) - \mathsf{TT}(G_1, \dots, G_d)\Big\|_F < \varepsilon \Big\|\mathsf{TT}(G_1, \dots, G_d)\Big\|_F$
    }{
        $\widehat{\mathcal{I}}_1,\dots, \widehat{\mathcal{I}}_{d-1}
            \gets \mathcal{I}_1,\dots,\mathcal{I}_{d-1} $ extended with random multi-indices\;
        $H_1,\dots, H_d, \mathcal{J}_1, \dots, \mathcal{J}_{d-1} \gets \mathtt{TTCrossRightToLeft}(A_{X,f}, \widehat{\mathcal{I}}_1,\dots, \widehat{\mathcal{I}}_{d-1}, \varepsilon_{\mathrm{eff}})$\;
        $\widehat{\mathcal{J}}_1,\dots, \widehat{\mathcal{J}}_{d-1} \gets \mathcal{J}_1,\dots,\mathcal{J}_{d-1} $ extended with random multi-indices\;
        $G_1,\dots, G_d, \mathcal{I}_1,\dots,\mathcal{I}_{d-1}
            \gets \mathtt{TTCrossLeftToRight}(A_{X,f}, \widehat{\mathcal{J}}_1,\dots, \widehat{\mathcal{J}}_{d-1}, \varepsilon_{\mathrm{eff}})$\;
    }
    \Return $G_1,\dots,G_d$\;
 \caption{\texttt{TTCross} constructs a TT approximation of a tensor (Section~\ref{sec:tt-cross})}
\end{algorithm}

\begin{algorithm}\label{alg:cross-tt-pass}
\KwInput{(black-box) tensor $A\in\mathbb{C}^{n_1\times\dots\times n_d}$, sets of column multi-indices $\widehat{\mathcal{J}}_1,\dots,\widehat{\mathcal{J}}_{d-1}$,
approximation tolerance $\varepsilon$}
\KwOutput{TT-approximation of $A$ with cores $G_1, \dots, G_{d-1}$, sets of row multi-indices $\mathcal{I}_1,\dots,\mathcal{I}_{d-1}$}
    $\mathcal{I}_0 \gets \{()\}$ \tcp{empty multi-index}
    $r_0 \gets 1$\;
    \For {$k \gets 1,\dots,d-1$}{
        $\widehat{\mathcal{I}}_k \gets \big\{(i_1,\dots,i_k)~:~(i_1,\dots,i_{k-1})\in\mathcal{I}_{k-1}, i_k \in \{0,\dots,n_k-1\}\big\}$\;
        $\mathcal{I}_k, \mathcal{J}_k \gets \mathtt{MatrixCrossApproximate}\big(A(\widehat{I}_k, \widehat{\mathcal{J}}_k), \varepsilon/\sqrt{d-1}\big)$\label{alg:cross-TT-pass:matrix-cross}\;
        $C_k \gets A(\widehat{\mathcal{I}}_k,\mathcal{J}_k)\big(A(\mathcal{I}_k, \mathcal{J}_k\big)^+_\tau$\;
        $r_k \gets |\mathcal{I}_k|$\;
        Store $G_k \gets \mathsf{reshape}(C_k, r_{k-1} \times n_k \times r_k)$\;
     }
     $G_d \gets A(\mathcal{I}_{d-1}, :)$\;
     \Return $G_1,\dots,G_{d}$, $\mathcal{I}_1, \dots,\mathcal{I}_{d-1}$\;
 \caption{\texttt{TTCrossLeftToRight} performs left-to-right pass of TT-cross approximation (Section~\ref{sec:tt-cross})}
\end{algorithm}

\begin{algorithm}\label{alg:matrix-cross}
\KwInput{(black-box) matrix $M\in\mathbb{C}^{m \times n}$,
approximation tolerance $\widehat{\varepsilon}$}
\KwOutput{Row and column indices $\mathcal{I}, \mathcal{J}$ such that $M(:,\mathcal{J})M(\mathcal{I},\mathcal{J})^{-1}M(\mathcal{I},:)$ approximates $M$}
    $M', j_1, r \gets M, 0, 0$\;    
    \Repeat{
        $ \mu \sqrt{(m-r)(n-r)} < \widehat{\varepsilon} \|M-M'\|_F$\label{alg:matrix-cross-stopping}
    }{
        $r \gets r + 1$\;
        $i_r \gets \arg\max_i |M'(i,j_r)|$\;
        $j_{r+1} \gets \arg\max_{j} |M'(i_r,j)|$\;
        $\mu \gets |M'(i_r,j_r)|$\;
        $M' \gets M' - M'(:,j_r)(M'(i_r,j_r))^{-1}M'(i_r,:)$\label{alg:matrix-cross:update}\;
    }
    \Return $\{i_1,\dots, i_r\}, \{j_1,\dots,j_r\}$\;
 \caption{\texttt{MatrixCrossApproximate} computes a cross approximation of a matrix (Section~\ref{sec:matrix-cross})}
\end{algorithm}

\subsection{Cross approximation of matrices}\label{sec:matrix-cross}
In this section we describe methods to approximate a matrix when we are allowed to access only a small fraction of its elements.
There is a well-known result that any matrix $M \in \mathbb{C}^{m \times n}$ of rank $r$ satisfies
\[
M = M(:, \mathcal{J})M(\mathcal{I}, \mathcal{J})^{-1}M(\mathcal{I}, :)
\]
where $\mathcal{I}$ and $\mathcal{J}$ are any size-$r$ row and column index sets such that submatrix at their intersection is nonsingular.
A more subtle result is that a matrix with approximate rank $r$ can be sometimes approximated with matrix $M(:, \mathcal{J}) G M(\mathcal{I}, :)$ for some $r \times r$ matrix $G$.
The most intriguing problem here is the choice of $r$ rows and columns.
It turns out that it makes sense to choose the submatrix with maximum volume (modulus of the determinant) among all the $r\times r$ submatrices: in~\cite{Goreinov2001} it was proven
that if $M(\mathcal{I}, \mathcal{J})$ is such a submatrix and moreover is nonsingular, then
\begin{equation}\label{eq:maxvol-bound}
	\left\|M - M(:, \mathcal{J})M(\mathcal{I}, \mathcal{J})^{-1}M(\mathcal{I}, :)\right\|_C \le (r + 1) \sigma_{r+1}(M).
\end{equation}
Here $\|M\|_C$ denotes Chebyshev norm ($\|M\|_C=\max_{i,j}|M(i,j)|$) and $\sigma_{r+1}(M)$ the $(r+1)$-th singular value of $M$.
The inversion of the matrix $M(\mathcal{I}, \mathcal{J})$ can, however, be numerically unstable if the latter is ill-conditioned.
Actually, the pseudoinverse can be used instead of the inverse:
in~\cite{Goreinov1997} it was proven that if again the submatrix $M(\mathcal{I}, \mathcal{J})$ has maximum volume and is nonsingular then
\[
\left\|M - M(:, \mathcal{J})\big(M(\mathcal{I}, \mathcal{J})\big)^+_\tau M(\mathcal{I}, :)\right\|_2 \le c (\sigma_{r+1}(M) + \tau)\sqrt{1 + r(\max\{m, n\} - r)},
\]
where $c$ and $\tau$ are positive constants and $G^+_\tau$ denotes $\tau$-pseudoinverse of $G$ (if a matrix $B$ has singular decomposition $B = U\Sigma V^*$, and $\Sigma_\tau$ coincides with $\Sigma$ except the elements smaller than $\tau$ are set to zero, then $\tau$-pseudoinverse of $B$ is the pseudoinverse of $U\Sigma_\tau V^*$).

The important practical question is how to find the submatrix of maximal volume.
Alas, this problem is NP-hard~\cite{Civril2009}, so we usually are looking for a submatrix with ``sufficiently'' large volume.
One approach was proposed in~\cite{Bebendorf2000} and it is used in the current implementation.
The idea is to construct approximations to $M$ with increasing ranks.
We start with a random column index $j_1$, choose the row with index $i_1$ satisfying $i_1 = \arg \max_i |M(i, j_1)|$ and use $M_1 = M(:,j_1)(M(i_1,j_1))^{-1}M(i_1,:)$ as the first approximation.
These steps are repeated for the residual matrix $M-M_1$, for the next residual etc., until the latest residual $M-M_1-\dots-M_r$ is small enough or the limit for $r$ is reached.
The indices $\mathcal{I}=\{i_1,\dots,i_r\}$ and $\mathcal{J} = \{j_1, \dots, j_r\}$ often yield a good submatrix.

The Algorithm~\ref{alg:matrix-cross} presents a more rigorous description of the method.
There are two nontrivial moments in it.
First, the line~\ref{alg:matrix-cross:update} looks like we need to compute and store the whole $m \times n$ matrix.
In fact, $M'$ is a black-box matrix and its elements are computed basing on the corresponding elements of $M$ and rank-1 matrices $M_k$ (not present explicitly in pseudocode).
Second, the stopping criterion in line~\ref{alg:matrix-cross-stopping} may also look weird.
Actually, the left part of it is an estimate of $\|M'\|_F$ as there are not more than $(m-r)(n-r)$ nonzero elements in this matrix and its largest in modulus element does not (hopefully) strongly exceed the value of $\mu$.
The value $\|M-M'\|_F$ in the right part of the comparison is our current rank-$r$ approximation of $M$ and can be computed efficiently.
Indeed, if we write down the rank-1 matrices $M_k$ as $M_k=u_kv_k^*$ and denote $U := [u_1, \dots,u_r]$ and $V = [v_1,\dots,v_r]$, then 
\[
\|M-M'\|_F^2 = \|M_1 + \dots + M_r\|_F^2 = \|UV^*\|_F^2 = \mathrm{Tr}(VU^*UV^*) = \mathrm{Tr}(U^*UV^*V).
\]
Matrices $U^*U$ and $V^*V$ have size $r \times r$ and can be computed in time $O((m+n)r^2)$.

\subsection{TT cross approximation}\label{sec:tt-cross}
In this section we give the details of algorithms~\ref{alg:tt-cross} and~\ref{alg:cross-tt-pass} for computing the TT approximation of a black-box tensor.
Suppose we are given $ d-1 $ sets of multi-indices
\begin{align*}
& \widehat{\mathcal{J}}_1 = \left\{\left(i_2 ^ {(1)}, \dots, i_d ^ {(1)}\right), \dots, \left(i_2 ^ {(R_1)}, \dots, i_d ^ {(R_1)}\right) \right\},
         \cdots,
    \widehat{\mathcal{J}}_{d-1} = \left\{i_d ^ {(1)}, \dots, i_d ^ {(R_ {d-1})} \right\}.
\end{align*}
Let us consider the $n_1\times R_1$ matrix $A(:, \widehat{\mathcal{J}}_1)$ and choose its submatrix $ \widehat A_1  = A(\mathcal{I}_1,\mathcal{J}_1) \in \mathbb{C} ^ {r_1 \times r_1} $ of large volume as described in Section~\ref{sec:matrix-cross}.
Also store the corresponding columns as $C_1 := A(:,\mathcal{J}_1)$.
Next, let us take the subtensor $A(\mathcal{I}_1, :,\dots, :)$ of size $r_1\times n_2 \times\dots\times n_d$
and consider the new tensor 
\[
B_1:= \mathsf{reshape}\left(A(\mathcal{I}_1, :,\dots, :), r_1n_2 \times n_3 \times \dots \times n_d\right).
\]
We continue in the same vein with the tensor $B_1$ and index sets $\widehat{\mathcal{J}}_2, \dots, \widehat{\mathcal{J}}_ {d-1} $.
Eventually for each $k=1,\dots,d-1$ the matrices $\widehat{A}_k\in\mathbb{C}^{r_k\times r_k}$ and $C_k \in \mathbb{C}^{r_{k-1}n_k\times r_k}$ will be constructed.
On the last step we work with 2-dimensional tensor $B_{d-2}$ of size $r_{d-2}n_{d-1}\times n_d$.
Now along with matrices $C_{d-1}$ and $\widehat{A}_{d-1}$ we store the row matrix $C_d := B_{d-2}(\mathcal{I}_{d-1}, :)$.
In the paper~\cite{Oseledets2010} it is proven that if
the numbers $r_1, \dots, r_{d-1}$ happen to coincide with the TT-ranks of $A$ and all of $\widehat{A}_k$ are nonsingular, then $A$ admits TT-decomposition with cores
\begin{align*}
G_1 & =C_1 \widehat A_1 ^ {- 1}, \\
G_k & = \mathsf{reshape}(C_k \widehat A_k^{-1}, r_{k-1}\times n_k \times r_k),~~k = 2, \dots, d-1, \\
G_d & = C_d.
\end{align*}

However, we are usually interested in low-rank approximations so our $r_k$ are likely to be smaller than actual TT-ranks of $A$.
Moreover, as discussed in Section~\ref{sec:matrix-cross}, it is more robust to use $\tau$-pseudoinverse of $\widehat{A}_k$ instead of $\widehat{A}_k^{-1}$.
Thus we construct only an approximate TT-decomposition and beside that its quality depends on the choice of index sets $\widehat{\mathcal{J}}_k$.
So to find a good approximation we need to repeat the whole procedure several times.
We start with arbitrary sets $ \widehat{\mathcal{J}}_k$ and after the execution of the above algorithm we obtain multi-index sets of rows, namely $\mathcal{I}_1, \dots, \mathcal{I}_{d-1}$, $|\mathcal{I}_k| = r_k$.
These sets are extended with random multi-indices to get the sets $\widehat{\mathcal{I}}_1, \dots, \widehat{\mathcal{I}}_{d-1}$, $|\widehat{\mathcal{I}}_k| = R_k$.
Now we run the same algorithm ``from right to left'', i.e. we start with $A(\widehat{\mathcal{I}}_{d-1}, :)$ and choose a submatrix of large volume in it.
The we proceed for $k = d-2, \dots, 1$, finally obtaining new approximation to $A$.

Assume that after the $j$-th pass of the above algorithm we obtain the TT-decomposition of the approximation $A^{(j)}$.
How can we estimate its closeness to $A$ or at least the convergence of the process?
In fact, we can efficiently compute $\|A^{(j)} - A^{(j-1)}\|_F$ to compare it with prescribed tolerance:
\[
\left\|A^{(j)} - A^{(j-1)}\right\|_F^2 = \left\langle A^{(j)}, A^{(j)}\right\rangle + \left\langle A^{(j-1)}, A^{(j-1)}\right\rangle - 2\mathrm{Re}\left\langle A^{(j)}, A^{(j-1)}\right\rangle.
\]
Each dot product can be computed efficiently using the TT-decompositions of $A^{(j-1)}$ and $A^{(j)}$ as described in~\cite{Oseledets2011}.

Note that for the sake of more straightforward implementation in the pseudocode of Algorithm~\ref{alg:cross-tt-pass} we use row multi-indices $\mathcal{I}_k$ instead of tensors $B_k$.
Also notice that in line~\ref{alg:cross-TT-pass:matrix-cross} we call subroutine \texttt{MatrixCrossApproximate} with $\widehat{\varepsilon} = \varepsilon/\sqrt{d-1}$.
The idea is that if we use $\widehat{\varepsilon}$-approximations for the unfolding matrices, the inequality~\eqref{eq:unfolding-approximation} lets us hope that the overall error will be bounded by $\widehat{\varepsilon}\sqrt{d-1} = \varepsilon$.

\subsection{Choosing approximation tolerance}\label{sec:tolerance}
This section describes our solution to the problem of choosing the approximation tolerance in Algorithm~\ref{alg:tt-cross}.
On the one hand, this choice can be made by the user (a rough estimate of the desired integration error can be used; for more formal analysis refer to Section~\ref{sec:analysis}).
On the other hand, sometimes the user is limited by computational resources available for the evaluation of an integral.
In such a situation the Monte-Carlo methods are more flexible as one can straightforwardly provide the limit on the number of function evaluations.
In the TT integration approach, however, all the function evaluations are performed inside the \texttt{MatrixCrossApproximate} subroutine.
If at some moment we find out that the limit is reached, we must abort immediately thus not obtaining a high-quality approximation of the current matrix.
Moreover, if this happens (in one of the executions of line~\ref{alg:cross-TT-pass:matrix-cross} of Algorithm~\ref{alg:cross-tt-pass}), we cannot proceed with the current TT-cross pass and are doomed to use the previous approximation.
And what if it was the first pass?

As you can see, the ``hard limit'' approach does not work for our method.
Instead, we propose another strategy.
It is based on the observation that for many real-life tensors the singular values of the unfolding matrices decay exponentially, i.e. \[
\sigma_\ell(A_k) \approx \alpha_k e^{-\beta_k \ell},~\ell=1,\dots,L_k,
\]
where $\alpha_k, \beta_k$ are some positive constants and $L_k$ is the smallest size of $A_k$.
If the rank-$r_k$ approximation to $A_k$ found in \texttt{MatrixCrossApproximate} is close enough to the optimal rank-$r_k$ approximation, we can write
\[
\widehat{\varepsilon}^2 \approx \sum_{\ell=r_k+1}^{L_k}(\sigma_\ell(A_k))^2 \approx \frac{\alpha_k^2}{2\beta_k}e^{-2\beta_k r_k}
\Rightarrow
r_k \approx -\beta_k\ln \widehat{\varepsilon} + \gamma_k.
\]
As we have mentioned, all tensor element evaluations happen in the \texttt{MatrixCrossApproximate} subroutine.
The matrix passed to the subroutine on the $k$-th step has size $r_{k-1}n_k \times R_k$.
All the elements of a cross of width $r_k$ (i.e. exactly $r_{k-1}n_kr_k + (R_k-r_k)r_k$ ones) are evaluated.
If we assume for simplicity that $R_k-r_k \ll r_{k-1}n_k$, we can write for the total number of element evaluations $N$:
\[
N \approx \sum_{k=1}^{d-1}n_kr_{k-1}r_k \approx (\ln\widehat{\varepsilon})^2\Big(\sum_{k=1}^{d-1} n_k\beta_{k-1}\beta_k\Big) + \sum_{k=1}^{d-1}\gamma_k.
\]
So if we ignore the sum of $\gamma_k$, we can assume that the number of evaluations is proportional to the squared logarithm of the tolerance.

Now, our heuristic works as follows (lines~\ref{alg:tt-cross-test} and~\ref{alg:tt-cross:effective} of Algorithm~\ref{alg:tt-cross}).
First, we run the Algorithm~\ref{alg:cross-tt-pass} with some ``test'' tolerance $\varepsilon_{\mathrm{test}}$ (e.g.~$0.01$).
Let us assume it has required $N_{\mathrm{test}}$ evaluations of tensor elements.
If the ``soft limit'' specified by the user is $N_{\mathrm{lim}}$ we will run the following passes with ``effective tolerance'' $\varepsilon_{\mathrm{eff}}$ which is computed basing on the formula 
\[
\frac{N_{\mathrm{lim}} - N_{\mathrm{test}}}{N_{\mathrm{test}}}
=
\frac{(\ln \varepsilon_{\mathrm{eff}})^2}{(\ln \varepsilon_{\mathrm{test}})^2}.
\]

\subsection{Functions with singularities}\label{sec:singularities}
It is well known that the numerical integration of functions with singularities (i.e. unbounded in some point of integration domain) give large errors.
There are several approaches to solve this problem (see e.g.~\cite{Kirkup2019}).
One of the most powerful is the variable substitution.
Let us substitute $x$ with some monotonic function $g(t)$ (here $g^{-1}(x)$ denotes the inverse function):
\[
\int_0^1 f(x)dx = \int_{g^{-1}(0)}^{g^{-1}(1)} f(g(t))g'(t)dt.
\]
After the transformation any quadrature rule (e.g. Gaussian) can be applied.
Note, however, that there is no need to apply it to the function $f(g(t))g'(t)$.
Instead, the quadrature rule may be transformed accordingly, i.e. instead of the nodes $t_i$ and weights $w_i$ we can use the nodes $g(t_i)$ and weights $w_ig'(t_i)$ and apply the new rule directly to the function $f$.
The advantage of this approach is most clearly seen when it used inside the TT integration and saves quite a lot of extra computations.

The appropriate function $g(t)$ is chosen basing on the location and type of the singularity.
As an example, consider logarithmic singularity at zero: $f(x) = O(\ln x)$ as $x \to 0$ (i.e. $f(x)$ approaches negative infinity asymptotically not faster than $\ln x$).
In this case the power transformation $x = t^p$, $p>1$ can be used.
Indeed, if we apply the transformed integral is 
\[
\int_0^1 f(t^p)pt^{p-1}dt = p\int_0^1 O(\ln(t)t^{p-1})dt
\]
and now the integrand is regular (i.e. does not have singularities) as for $\alpha >0$ the function $\ln(t)t^\alpha$ tends to $0$ as $t \to 0$.
Other useful substitutions include $\tanh$-$\sinh$~\cite{Mori2001}
($
x = \tanh(\frac{1}{2}\pi\sinh t)
$)
and erf~\cite{Kirkup2019}
($
x = \frac{1}{2}(1-\mathrm{erf}(t))
$)
ones.

The substitution approach can be easily adopted for TT integration technique if there is only one singularity in the multivariate integration domain $[0,1]^d$.
For simplicity let us assume that it is located at zero.
Then for each $k=1,\dots,d$ we choose a quadrature rule with nodes $x_k^{(i_k)}$ and weights $w_k^{(i_k)}$ such that under the corresponding transforms $x_k=g_k(t_k)$ the new integrand \[
f(x_1,\dots,x_{k-1},g_k(t_k),x_{k+1},\dots,x_d)g'_k(t_k)
\]
viewed as a function of $t_k$ is regular for any $x_1,\dots,x_{k-1},x_k,\dots,x_d$.

\subsection{Error analysis}\label{sec:analysis}
The integration algorithm introduces two kinds of errors: the approximation error (tensor $A'$ with hopefully moderate TT-ranks differs from $A_{X, f}$) and the numerical integration error ($S_{X,W}(f)$ differs from $I(f)$).
In this section we estimate these two type of errors separately and eventually come to the bound of $|I(f) - S'|$.

For $k = 1, \dots, d-1$ we will denote by $I_k(x_{k+1}, \dots, x_d)$ the function
\[
I_k(x_{k+1}, \dots, x_d) = \int_0^1\cdots\int_0^1 f(x_1, \dots, x_d)dx_1\cdots dx_k.
\]
Naturally, we define $I_0(x_1, \dots, x_d) \equiv f(x_1, \dots, x_d)$ and $I_d \equiv I(f)$.

\begin{lemma}\label{lm:integration-error}
Let us assume that for all $k = 1, \dots, d$ two properties hold:
\begin{enumerate}
    \item the quadrature rule with nodes $x_k^{(i_k)}$ and weights $w_k^{(i_k)}$ has nonnegative weights and is exact for constant functions, i.e.
    \begin{equation}\label{eq:good-quadrature}
    w_k^{(i_k)} \ge 0,~~1 = \int_0^1 1 \cdot dx = \sum_{i_k=0}^{n_k-1} w_k^{(i_k)};
    \end{equation}
    \item for all $x_{k+1}, \dots, x_d \in [0,1]$ we have
\begin{equation}\label{eq:partial-integral}
    \left|\int_0^1 I_{k-1}(x_{k}, \dots, x_d)dx_k - \sum_{i_k=0}^{n_k-1} I_{k-1}\left(x_k^{(i_k)}, x_{k+1}, \dots, x_d\right)w_k^{(i_k)}\right| \le \varepsilon_{\mathrm{int}}.
\end{equation}
\end{enumerate}
Then
\begin{equation}
    \left|I(f) - S_{X,W}(f)\right| \le d \varepsilon_{\mathrm{int}}.
\end{equation}
\end{lemma}
\begin{proof}
We will use induction to prove the following inequality for all $k=1, \dots, d$ and all $x_{k+1}, \dots, x_d \in [0,1]$:
\begin{equation}\label{integration-induction}
\left| I_k(x_{k+1}, \dots, x_d) - \sum_{i_1,\dots,i_k} f\left(x_1^{(i_1)}, \dots, x_k^{(i_k)}, x_{k+1}, \dots, x_d\right)w_1^{(i_1)}\dots w_k^{(i_k)}  \right| \le k \varepsilon_{\mathrm{int}}.
\end{equation}
The base case corresponds to $k=1$ and follows immediately from \eqref{eq:partial-integral} as by definition $I_{0}(x_1, \dots, x_d) \equiv f(x_1, \dots, x_d)$.

Now proceed to the induction step.
To estimate the considered difference we add and subtract the same value and then use the triangle inequality:
\begin{align*}
\left| I_k(x_{k+1}, \dots, x_d) - \sum_{i_1,\dots,i_k} f\left(x_1^{(i_1)}, \dots, x_k^{(i_k)}, x_{k+1}, \dots, x_d\right)w_1^{(i_1)}\dots w_k^{(i_k)}  \right|
&\le \\ \le 
\left| I_k(x_{k+1}, \dots, x_d) - 
\sum_{i_k=0}^{n_k-1} I_{k-1}\left(x_k^{(i_k)}, x_{k+1}, \dots, x_d\right)w_k^{(i_k)}\right|
&+ \\ +
\left|\sum_{i_k=0}^{n_k-1} I_{k-1}\left(x_k^{(i_k)}, x_{k+1}, \dots, x_d\right)w_k^{(i_k)} -
\sum_{i_1,\dots,i_k} f\left(x_1^{(i_1)}, \dots, x_k^{(i_k)}, x_{k+1}, \dots, x_d\right)w_1^{(i_1)}\dots w_k^{(i_k)}  \right|.
\end{align*}
The first summand is exactly the left part of \eqref{eq:partial-integral} and thus not greater than $\varepsilon_{\mathrm{int}}$.
The second summand can be estimated as follows:
\begin{align*}
M := \left|
\sum_{i_k=0}^{n_k-1} I_{k-1}\left(x_k^{(i_k)}, x_{k+1}, \dots, x_d\right)w_k^{(i_k)}
-
\sum_{i_1,\dots,i_k} f\left(x_1^{(i_1)}, \dots, x_k^{(i_k)}, x_{k+1}, \dots, x_d\right)w_1^{(i_1)}\dots w_k^{(i_k)} 
\right|
&= \\ =
\left|
\sum_{i_k=0}^{n_k-1} \left(
I_{k-1}\left(x_k^{(i_k)}, x_{k+1}, \dots, x_d\right)
-
\sum_{i_1,\dots,i_{k-1}} f\left(x_1^{(i_1)}, \dots, x_k^{(i_k)}, x_{k+1}, \dots, x_d\right)w_1^{(i_1)}\dots w_{k-1}^{(i_{k-1})}
\right) w_k^{(i_k)}
\right|
&\le \\ \le
\sum_{i_k=0}^{n_k-1}
\left|
I_{k-1}\left(x_k^{(i_k)}, x_{k+1}, \dots, x_d\right)
-
\sum_{i_1,\dots,i_{k-1}} f\left(x_1^{(i_1)}, \dots, x_k^{(i_k)}, x_{k+1}, \dots, x_d\right)w_1^{(i_1)}\dots w_{k-1}^{(i_{k-1})}
\right|
w_k^{(i_k)}.
\end{align*}
By induction hypothesis and the property~\eqref{eq:good-quadrature} of the quadrature rule the inequality chain can be continued:
\begin{align*}
    M \le
\sum_{i_k=0}^{n_k-1}
(k-1)\varepsilon_{\mathrm{int}} w_k^{(i_k)} \le (k-1)\varepsilon_{\mathrm{int}}\sum_{i_k=0}^{n_k-1}
w_k^{(i_k)} = (k-1)\varepsilon_{\mathrm{int}}.
\end{align*}
The induction step is thus proven. 
The conclusion of the lemma follows from \eqref{integration-induction} for $k = d$.
\end{proof}

Let us discuss the applicability of the proven lemma to some classes of functions.
How restrictive is the property \eqref{eq:partial-integral}?
We will consider two examples.
First, if we use the $s$-point Gaussian quadrature rule and the function $f(x_1, \dots, x_n)$ is a polynomial where each variable $x_k$ has degree not greater than $2s-1$, then the quadrature rule gives the exact result: $I(f) = S_{X, W}(f)$.
    Indeed, for every monomial $x_1^{p_1}\cdots x_d^{p_d}$
    \[
    I_{k-1}(x_{k}, \dots, x_d) = \frac{1}{(1+p_1)\cdots(1+p_{k-1})} x_{k}^{p_{k}}\cdots x_{d}^{p_d}
    \]
    and
    \begin{align*}
    \int_0^1 I_{k-1}(x_{k}, \dots, x_d)dx_k 
    = 
    \frac{1}{(1+p_1)\cdots(1+p_{k-1})} x_{k+1}^{p_{k+1}}\cdots x_{d}^{p_d}\int_0^1 x_k^{p_k} dx_k
    = \\ =
    \frac{1}{(1+p_1)\cdots(1+p_{k-1})} x_{k+1}^{p_{k+1}}\cdots x_{d}^{p_d}\sum_{i_k=0}^{n_k-1} \left(x_k^{(i_k)}\right)^{p_k} w_k^{(i_k)}
    =
    \sum_{i_k=0}^{n_k-1} I_{k-1}\left(x_k^{(i_k)}, x_{k+1}, \dots, x_d\right)w_k^{(i_k)},
    \end{align*}
    where the second equality follows from the fact that the Gaussian quadrature rule is exact for polynomials of degree not greater than $2s-1$ (see e.g.~\cite[Thm. 16.5.1]{Tyrtyshnikov2012}).
    By linearity the equality can be extended to the integral of the polynomial $f(x_1,\dots, x_n)$.
    
In the second example we again use the $s$-point Gaussian quadrature rule and the function $f$ has $2s$ continuous partial derivatives w.r.t. each variable with a known bound:
\[
    \left|\frac{\partial^{2s}}{\partial x_k^{2s}} f(x_1, \dots, x_d) \right|
    \le 
    D.
\]
In this case the following inequality holds:
\[
    \left| I(f) - S_{X,W}(f) \right| \le \frac {\left(s!\right)^{4}}{(2s+1)\big(\left(2s\right)!\big)^{3}}\cdot Dd.
\]
Indeed, by Leibniz integral rule (see e.g.~\cite{Protter2012}) we can move differentiation under the integral sign and obtain the following estimate:
\begin{align*}
    \left|\frac{\partial^{2s}}{\partial x_k^{2s}} \int_0^1\cdots\int_0^1 f(x_1, \dots, x_d)dx_1\cdots dx_{k-1}\right|
    &=
    \left|\int_0^1\cdots\int_0^1 \frac{\partial^{2s}}{\partial x_k^{2s}} f(x_1, \dots, x_d)dx_1\cdots dx_{k-1}
    \right|
    \le \\ &\le
    \int_0^1\cdots\int_0^1 \left|\frac{\partial^{2s}}{\partial x_k^{2s}} f(x_1, \dots, x_d)\right| dx_1\cdots dx_{k-1}
    \le
    D.
\end{align*}
It is known (e.g. see~\cite[\S 5.2]{Kahaner1989}) that for some $\xi\in(0,1)$ the integration error $\varepsilon_{\mathrm{int}}$ from \eqref{eq:partial-integral} can be bounded by
\[
    \frac {\left(s!\right)^{4}}{(2s+1)\big(\left(2s\right)!\big)^{3}} \left|\frac{\partial^{2s}}{\partial x_k^{2s}}I_{k-1}(\xi,x_{k+1}, \dots, x_d)\right| \le   \frac {\left(s!\right)^{4}}{(2s+1)\big(\left(2s\right)!\big)^{3}}\cdot D.
\]
The required bound for $|I(f)-S_{X,W}(f)|$ follows immediately.

It can be seen from the above examples that the property~\eqref{eq:partial-integral} is not extremely restrictive and can sometimes be derived from the corresponding one-dimensional results.

Now we can prove the main theorem.

\begin{theorem}
If the conditions of lemma~\ref{lm:integration-error} hold and tensor $A' \in \mathbb{C}^{n_1\times\cdots\times n_d}$ satisfies the inequality
\[
\left\|A_{X,f}-A'\right\|_F \le \varepsilon_{\mathrm{appr}},
\]
then
\[
\left|I(f) - \langle \overline{A'}, W \rangle\right| \le d\varepsilon_{\mathrm{int}} + \varepsilon_{\mathrm{appr}}.
\]
\end{theorem}
\begin{proof}
We start by using the triangle property of modulus:
\begin{align}
\left|I(f) - \langle \overline{A'}, W \rangle\right|
&\le
\left|I(f) - \langle \overline{A_{X,f}}, W \rangle\right| + \left|\langle \overline{A_{X,f}}, W\rangle - \langle \overline{A'}, W \rangle\right|  \nonumber
= \\ &=
\left|I(f) - \langle \overline{A_{X,f}}, W \rangle\right| + \left|\langle \overline{A_{X,f} - A'}, W \rangle\right|.
\label{eq:triangle}
\end{align}
The first term has been estimated in lemma~\ref{lm:integration-error} by $d\varepsilon_{\mathrm{int}}$.
The second summand can be bounded using the Cauchy--Bunyakovsky--Schwarz inequality:
\[
\left|\langle \overline{A_{X,f} - A'}, W \rangle\right| \le \sqrt{\left|\langle \overline{A_{X,f} - A'}, \overline{A_{X,f} - A'}\rangle\right|}\cdot \sqrt{\left|\langle W, W \rangle\right|} = \|A_{X,f} - A'\|_F \|W\|_F.
\]
It suffices to estimate the value of $\|W\|_F$.
Notice that by~\eqref{eq:good-quadrature} all the quadrature weights satisfy $0 \le w_k^{(i_k)} \le 1$, thus we can write
\begin{align*}
    \|W\|_F^2
    &=
    \sum_{i_1,\dots,i_d}\prod_{k=1}^d \left(w_k^{(i_k)}\right)^2
    =
    \prod_{k=1}^d \sum_{i_k=0}^{n_k-1} \left(w_k^{(i_k)}\right)^2
    \le
    \prod_{k=1}^d \sum_{i_k=0}^{n_k-1} w_k^{(i_k)}
    =
    \prod_{k=1}^d 1
    =
    1.
\end{align*}
\end{proof}

\section{Implementation and experiments}\label{sec:implementation}
We have implemented the described algorithm in C++.
The source code is open and distributed under the MIT license.
The repository resides at \url{https://bitbucket.org/vysotskylev/c-tt-library}.
The implementation began as a fork of tensor-train library by Dmitry Zheltkov (\url{https://bitbucket.org/zheltkov/c-tt-library}).
The only serious dependency is OpenBLAS library (\url{https://www.openblas.net}) used for matrix operations.

Several examples of integration invocation are provided in \texttt{tests/test\_integration.cpp}.
The key function is \texttt{TTIntegrate}.
It takes as parameters the integrand, the domain and configuration.
The integrand is a functor with parameters \texttt{int count, double* points, double* result}.
The integrand must be evaluated in \texttt{count} points with coordinates passed in \texttt{points} array (which must be of size \texttt{count * dimensionality}) and the results must be written to \texttt{result} array.
The integration domain is supposed to be an axis-aligned parallelepiped which is specified through the \texttt{box} parameter.
The \texttt{config} parameter allows the user to choose the limit on the number of function evaluations, prescribed accuracy, type of quadrature etc.
For the most recent information on integration parameters and invocation refer to \texttt{README.md} file in the repository.

By default the library uses multi-threading for matrix operations.
However, the calls to the functor are guaranteed to be sequential.
The user is encouraged to use parallel computations to evaluate the function in provided points.
The library strives to pass to the functor batches large enough to keep multiple CPU cores (or even a GPU) busy.
In the applications where multiple integrals are to be computed the parallelization is even more straightforward.

As a model example for integration consider the following integral with a singularity at zero:
\[
\int_0^1\dots\int_0^1 \ln(x_1\cdots x_d) dx_1\cdots dx_d.
\]
The analytical answer is easily computed to be $-d$.
We compare the proposed algorithm with Monte-Carlo type method called Vegas implemented in the Cuba library~\cite{Hahn2005}.
The plot of relative accuracy depending on the number of variables is shown at Figure~\ref{fig:compare}.
The power-3 transform and Gauss-Legendre quadrature were used for the TT integration.
Both algortihms were limited by 1 million function evaluations.

\begin{figure}
    \centering
    \includegraphics[width=0.8\linewidth]{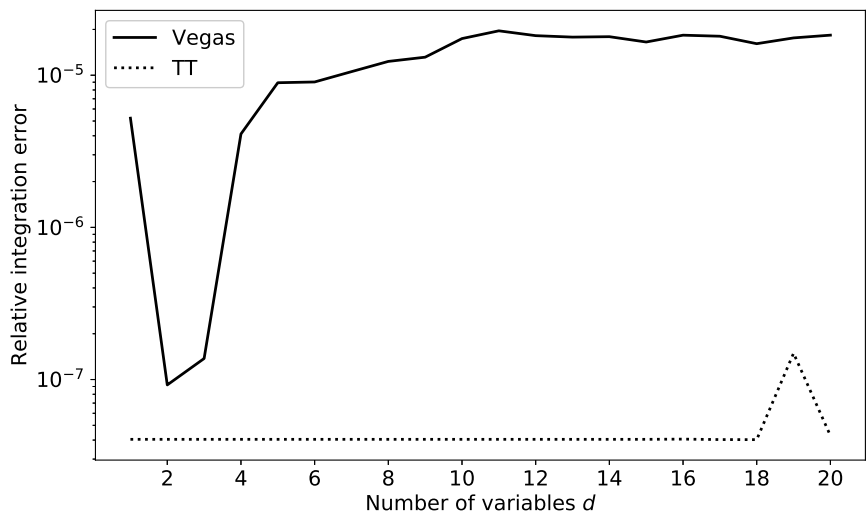}
    \caption{Comparison of TT and Vegas integration on the model example}
    \label{fig:compare}
\end{figure}

We added the new TT integrator as one of the integrators used in a development version of the FIESTA~\cite{Smirnov2015} program used to evaluate Feynman integrals numerically.
Currently there exists a problem to numerically verify some results for Feynman integrals\footnote{They appear as ingredients in quantum chromodynamics (QCD) calculations and in supersymmetric $N=4$ Yang-Mills theory within perturbation theory.
These Feynman integrals are found in many important functions of these theories, in particular, various anomalous dimensions and form factors.} obtained analytically.
Most of them were checked with the Vegas integrator, but we also used one of master integrals as a cross-check for the TT integrator proposed in this paper.
In the sector decomposition approach this Feynman integral is represented as a sum of 1208 multivariate integrals of dimensions 5 or 6.
These integrals were computed numerically using TT and Vegas integration. 
The analytical result obtained independently is a Laurent polynomial of variable $\varepsilon$ so we could compare the accuracy of numerical integration methods.
The formula for the polynomial is 
\[
\frac{1}{24}\varepsilon^{-3} + \frac{9}{16}\varepsilon^{-2} + \Big(\frac{451}{96} + \frac{5\pi^2}{72}\Big)\varepsilon^{-1}+ \frac{18165 + 540\pi^2 - 1040\cdot\zeta(3)}{576}
\]
(here $\zeta(s)$ denotes Riemann Zeta function).
The integrators (TT and Vegas) were allowed to evaluate the integrand in 1 million points.
For TT integration we used Gauss-Legendre quadrature with 13 nodes on each axis and performed power transform (see Section~\ref{sec:singularities}) of form $x = t^3$ as the integrand has a logarithmic singularity at zero.
The relative errors for each term are summarized in Table~\ref{table:errors}.

\begin{table}[]
\centering
\begin{tabular}{|l|l|l|l|l|}
\hline
      & $\varepsilon^{-3}$     & $\varepsilon^{-2}$   & $\varepsilon^{-1}$     & $\varepsilon^{0}$     \\ \hline
TT    & $27.3 \times 10^{-6}$  & $0.3 \times 10^{-6}$ & $0.06  \times 10^{-6}$ & $0.88 \times 10^{-6}$ \\ \hline
Vegas & $188.6 \times 10^{-6}$ & $7.1 \times 10^{-6}$ & $0.11 \times 10^{-6}$  & $3.96 \times 10^{-6}$ \\ \hline
\end{tabular}
\vspace{0.5cm}
\caption{Relative errors of numerical integration for a Feynman integral}
\label{table:errors}
\end{table}

The results are quite promising and show that in some circumstances the new integrator can outperform the Vegas algorithm.
We are working on a new release of the FIESTA program where the new TT integrator might become one of the key components.

\section{Acknowledgements}
The work is supported by Russian Ministry of Science and Higher Education, agreement No. 075-15-2019-1621.

\end{document}